\documentclass[11pt,reqno]{amsart}
\usepackage{amsfonts,amssymb,amsthm}
\usepackage{amsmath}
\usepackage[numbers,sort&compress]{natbib}
\usepackage[mathscr]{euscript}
\usepackage{graphicx}
\usepackage{booktabs,float}
\usepackage{color}
\usepackage{mathrsfs}
\usepackage{paralist}
\usepackage{hyperref}
\usepackage{framed}

\makeatletter
\newcommand{\leqnomode}{\tagsleft@true}
\newcommand{\reqnomode}{\tagsleft@false}
\makeatother

\newtheorem {theorem}{Theorem}[section]

\newtheorem {corollary}{Corollary}[section]

\newtheorem {lemma}{Lemma}[section]

\newtheorem {example}{Example}[section]
\newtheorem {definition}{Definition}[section]
\newtheorem {remark}{Remark}[section]

\newcommand{\ER}{\overline{\mathbb{R}}}

\newcommand{\bfz}{{\bf 0}}
\newcommand{\R}{{\mathbb R}}

\newcommand{\sph}{\mathbb{S}}
\newcommand{\N}{\mathbb{N}}
\newcommand{\B}{\mathbb{B}}
\newcommand{\cl}{{\rm cl\,}}

\def\ees{{\accent"5E e}\kern-.385em\raise.2ex\hbox{\char'23}\kern-.08em}
\def\EES{{\accent"5E E}\kern-.5em\raise.8ex\hbox{\char'23 }}
\def\ow{o\kern-.42em\raise.82ex\hbox{
   \vrule width .12em height .0ex depth .075ex \kern-0.16em \char'56}\kern-.07em}
\def\OW{O\kern-.460em\raise1.36ex\hbox{
\vrule width .13em height .0ex depth .075ex \kern-0.16em \char'56}\kern-.07em}

\title{Existence of Pareto Solutions for Vector Polynomial Optimization Problems with Constraints}

\author{Yarui Duan}
\address[Yarui Duan]{School of Mathematical Sciences, Soochow University, Suzhou 215006, China}
\email{dyrsuda@163.com}

\author{Liguo Jiao}
\address[Liguo Jiao]{Academy for Advanced Interdisciplinary Studies, Northeast Normal University,
Changchun 130024, Jilin Province, China}
\email{hanchezi@163.com; jiaolg356@nenu.edu.cn}

\author{Pengcheng Wu$^{*}$}
\address[Pengcheng Wu]{School of Mathematical Sciences, Soochow University, Suzhou 215006, China}
\email{pcwu0725@163.com}
\thanks{$^{*}$Corresponding Author}

\author{Yuying Zhou}
\address[Yuying Zhou]{School of Mathematical Sciences, Soochow University, Suzhou 215006, China}
\email{yuyingz@suda.edu.cn}

\date{\today}

\begin{document}

\begin{abstract}
This paper deals with a  vector polynomial optimization problem over a basic closed semi-algebraic set.
By invoking some powerful tools from real semi-algebraic geometry, we first introduce the concept called {\it tangency varieties};   obtain   the relationships of the Palais--Smale condition, Cerami condition, {\it M}-tameness, and properness related to the considered problem, in which the condition of Mangasarian--Fromovitz constraint qualification at infinity plays an essential role in deriving these relationships. At last
according to the obtained connections, we establish the existence of Pareto solutions to the problem in consideration  and   give some examples to illustrate our main findings.
\end{abstract}
\maketitle

\section{Introduction}\label{Sec:1}

Existence of optimal solutions to optimization problems is a rather important issue in in the study of optimization theory.
In the literature on vector optimization (among others), one can find a lot of papers dealing with the existence of different kinds of solutions to vector optimization problems; see, e.g., \cite{Borwein1983,Deng1998a,Deng1998b,Deng2010,Gutierrez2014,Huang2004,Kim2019,Kim2020} and the references therein.

Consider the following constrained vector {\em polynomial} optimization problem
\begin{align}\label{problem}
{\rm Min}_{\mathbb{R}^p_+}\;\big\{f(x)\,\colon \,x\in S\big\},\tag{VPO}
\end{align}
where $f(x):=(f_1(x), \ldots, f_p(x))$ is a real polynomial mapping, and
\begin{align}\label{SAset}
S:=\{x\in\R^n \colon g_i(x)=0,i=1,\ldots,l,\ h_j(x)\geq0,j=1,\ldots,m\}
\end{align}
is the feasible set of the problem~\eqref{problem}, in which, $g_i, i = 1, \ldots, l,$ and $h_j, j = 1, \ldots, m$ are all real polynomials.
As we will see from Definition~\ref{definitionSA} that $S$ is a closed semi-algebraic set.
Furthermore,   make the following assumption:
\begin{itemize}
\item[$\quad$] $\qquad\qquad\qquad\qquad\qquad\quad$ \framebox[1.07\width]{the feasible set $S$ is unbounded.}
\end{itemize}
Note also that the ``${\rm Min_{\R^p_+}}$" in the above problem~\eqref{problem} is understood in the vector sense, where a partial ordering is induced in the image space $\R^p,$ by the non-negative orthant $\R^p_+.$
The partial ordering says that $a \geq b,$ if $a - b \in \R^p_+,$ which can equivalently be written as $a_k \geq b_k,$ for all $k = 1, \ldots, p,$ where $a_k$ and $b_k$ stand for the $k$th component of the vectors $a$ and $b,$ respectively.

\subsection{Pareto values and solutions}
In what follows, we recall the Pareto values and Pareto soultions to the problem~\eqref{problem}.
Unless the classical literature on vector optimization (see, e.g., \cite{Ehrgott2005,Jahn2004,Luc1989,Sawaragi1985}), we will first introduce the Pareto values to the problem~\eqref{problem}, then give the definition of its Pareto solutions.
Let $f(S)$ be the image of the restrictive real polynomial mapping $f$ over $S$.
\begin{definition}\label{Pareto}{\rm
Let $y\in \,\mathrm{cl} f(S)$.
\begin{itemize}
\item[(i)] $y\in \R^p$ is called a {\em Pareto value} to the problem~\eqref{problem} if
		$$f(x)\notin y - (\mathbb{R}^p_+\setminus\{{\bfz}\}), \quad \forall x \in S,$$
		where $\bfz:=(0, \ldots, 0) \in \R^p$. The set of  all Pareto values to the problem~\eqref{problem} is denoted by $\mathrm{val}\,\eqref{problem}$.

\item[(ii)] $y\in \R^p$ is  called a {\em weak Pareto value} to the problem~\eqref{problem} if
		$$f(x)\notin y -{\rm int}\,\mathbb{R}^p_+, \quad \forall  x\in S.$$
		Denote $\mathrm{val}^w\,\eqref{problem}$ as  the set of  all weak Pareto values to the problem~\eqref{problem}.

\item[(iii)] $\bar x \in S$ is called a {\em Pareto solution} (resp., {\em weak Pareto solution}) to the problem~\eqref{problem} if $f(\bar x)$ is a Pareto value (resp., weak Pareto value) to the problem~\eqref{problem}. Denoted   $\mathrm{sol}\,\eqref{problem}$ (resp., $\mathrm{sol}^w\,\eqref{problem}$) as the set of all Pareto solutions (resp., weak Pareto solutions).
\end{itemize}
}\end{definition}
	
According to the above definitions, it is clear that $\mathrm{val}\,\eqref{problem}\subset \mathrm{val}^w\,\eqref{problem}$.

\begin{definition}\label{section}{\rm
Let $\Omega$ be a subset in $\R^p$ and $\bar y \in \R^p.$
The set $\Omega \cap (\bar y - \R^p_+)$ is said to be a {\it section} of $\Omega$ at $\bar y,$ and denoted by $[\Omega]_{\bar y}.$
The section $[\Omega]_{\bar y}$ is said to be {\it bounded} if and only if there is $\omega \in \R^p$ such that
$$[\Omega]_{\bar y} \subset \omega + \R^p_+.$$
}\end{definition}

\subsection{Backgrounds}
In this part, we will treat the problem~\eqref{problem} as a standard vector optimization problem (not necessarily under the polynomial setting).

Firstly, let us recall some results on the existence of Pareto solutions to the problem~\eqref{problem} in the case that the feasible set $S$ is nonempty and compact.
If in addition $f$ is $\R^p_+$-semicontinuous (see \cite[Definition 2.16]{Ehrgott2005}), then the existence of Pareto solutions to the problem~\eqref{problem}  was shown by Hartley \cite{Hartley1978} in 1978.
Later, Corley \cite{Corley1980} in 1980 proved also the existence of Pareto solutions to the problem~\eqref{problem}, if the image $f(S)$ is nonempty and $\R^p_+$-semicompact (see \cite[Definition 2.11]{Ehrgott2005}).
In 1983, it was observed by Borwein \cite[Theorem 1]{Borwein1983} that the condition ``the image $f(S)$ has at least one nonempty {\it closed} and {\it bounded} section" is a necessary and also {\it sufficient} condition for the existence of Pareto solutions to the problem~\eqref{problem}; see also \cite[Theorem 2.10]{Ehrgott2005}.
Clearly, the compactness of $S$ together with the continuity (or even semi-continuity) of   $f$ ensures the compactness of the image $f(S),$ in this case, the problem~\eqref{problem} admits at least one Pareto solution; see, e.g., \cite[Corollary 3.2.1]{Sawaragi1985}.

Now, we recall some existence results for the problem~\eqref{problem} in another case that the feasible set $S$ is not compact.
By assuming that the objective function is bounded from below and satisfies the so-called (PS)$_1$ condition,  H\`a \cite{Ha2006JMAA} proved  that   the problem~\ref{problem} has {\em weak} Pareto solutions in 2006 (see \cite[Theorem 4.1]{Ha2006JMAA}).
Later, by exploring  {\em quasiboundedness from below} and {\em refined subdifferential Palais--Smale condition}, Bao and Mordukhovich \cite{Bao2007,Bao2010MP} investigates some  vector optimization problems.
It is worth to note that   they established  only the existence of weak or relative  Pareto
solutions, but not  the existence of Pareto
solutions to the vector optimization problems in \cite{Bao2007,Bao2010MP,Ha2006JMAA}.

In order to obtain the results for existence of Pareto solutions, Lee et al. \cite[Theorem 3.1]{Lee2021} proved that the problem~\eqref{problem} admits a Pareto solution if and only if the image $f(S)$ of $f$ has a nonempty and bounded section for the case that $f$ is a convex polynomial mapping (each component of $f$ being convex polynomial), in which the celebrated existence results for scalar convex polynomial programming problems contributed by Belousov and Klatte \cite[Theorem 3]{Belousov2002} are applied.

Very recently, for the case that $S=\R^n$ and the image $f(\R^n)$ of a polynomial mappying $f$ has a bounded section,
Kim et al. \cite{Kim2019} investigated the existence of Pareto solutions to the problem~\eqref{problem} under
some novel conditions.

Furthermore, in order to investigate existence results in more general setting, by employing the theory of variational analysis and nonsmooth analysis (instead methods of semialgebraic geometry), Kim et al \cite{Kim2020} furtherly proved that nonconvex and nonsmooth vector optimization problems with locally Lipschitzian data have   Pareto efficient and Geoffrion-properly efficient solutions.
It is also worth mentioning that, Liu et al \cite{Liu2021} studied the solvability for a class of regular polynomial vector optimization problem without convexity, and interestingly even without semi-algebraic assumption for the feasible set $S$ (see \cite[Example 5.4]{Liu2021}).

\subsection{Our contributions}

In this paper, we will make the following contributions to the area of vector optimization with polynomials.
\begin{itemize}
\item[{\rm (i)}]We prove the existence of Pareto solutions to the {\it constrained} vector polynomial optimization problem~\eqref{problem} under some conditions.
Comparing with \cite{Lee2021}, we do not need any convexity assumptions in the problem~\eqref{problem}, and comparing with \cite{Kim2019}, we further consider the problem~\eqref{problem} over a closed (and unbounded) semi-algebraic set $S.$
\item[{\rm (ii)}] By constructing some suitable sets (that can be computed effectively) related to the problem~\eqref{problem}, we define the concepts concerning Palais--Smale condition, Cerami condition and $M$-tameness, and also establish some relationships between them (see Theorem~\ref{relationship1}).
All of these concepts play the important roles in establishing some sufficient conditions for the existence of Pareto solutions to the problem~\eqref{problem}.
\item[{\rm (iii)}]
It is worth emphasizing that, in Theorem~\ref{relationship1}, the Mangasarian--Fromovitz constraint qualification at infinity of $S$ (see Definition~\ref{regulity}) plays an essential role.
This significantly improves \cite[Proposition 3.2]{Kim2019}.
In order to highlight this observation, we construct an example (see Example~\ref{example0717}) to show that the assumption on Mangasarian--Fromovitz constraint qualification at infinity of $S$ cannot be dropped.
Besides, we also design several examples to illustrate some related terminologies and the obtained results.
\item[{\rm (iv)}] As results, we establish some sufficient conditions for the existence of Pareto solutions to the problem~\eqref{problem}.
The obtained results improve and extend \cite[Theorem 4.1]{Kim2019}, \cite[Theorem 4.1]{Ha2006JMAA}, \cite[Theorem 4]{Bao2007} and \cite[Theorem 4.4]{Bao2010MP}, in the polynomial setting.
\end{itemize}

\medskip

The rest of the paper is organized as follows.
In Sect.~\ref{Sec:2}, we recall some necessary tools from real semi-algebraic geometry.
In Sect.~\ref{Sec:3}, we introduce the concept of the tangency variety, which will be useful in the later, and its properties.
In Sect.~\ref{Sec:4}, we construct some suitable sets, by which, we establish some relationships between Palais--Smale condition, Cerami condition, $M$-tameness, and properness for the restrictive polynomial mappings.
Section~\ref{Sec:5} contains several   existence results of Pareto solutions to the problem~\eqref{problem}.
Finally, conclusions and further discussions are given in Sect.~\ref{Sec:7}.

\section{Preliminaries}\label{Sec:2}
Throughout this paper, we use the following notation and terminology.
Fix a number $n \in \N,$ $n \geq 1,$ and abbreviate $(x_1, x_2, \ldots, x_n)$ by $x.$
The space $\R^n$ is equipped with the usual scalar product $\langle \cdot, \cdot \rangle$ and the corresponding Euclidean norm $\| \cdot \|.$
The interior (resp., the closure) of a set $S$ is denoted by ${\rm int} S$ (resp., $\cl S$).
The closed unit ball in $\R^n$ is denoted by $\B^n.$
Let $\R^p_+:= \{y := (y_1, \ldots, y_p) \colon y_j \geq 0, \ j = 1, \ldots, p\}$ be the nonnegative orthant in $\R^p.$
The cone $\R^p_+$ induces the following partial order in $\R^p: a, b \in \R^p,$
$a \leq b$ if and only if $b - a \in \R^p_+.$
Besides, $\R[x]$ stands for the space of real polynomials in the variable $x.$
Let us recall some notion and results from semi-algebraic geometry (see,
e.g., \cite{RASS,Bochnak1998}).

\begin{definition}\label{definitionSA}{\rm
\begin{enumerate}
  \item[(i)] A subset of $\mathbb{R}^n$ is called a  {\em semi-algebraic} set if
	  it is a finite union of sets of the form
$$\{x \in \mathbb{R}^n \colon \varrho_i(x) = 0, i = 1, \ldots, k;\ \varrho_i(x) > 0,
i = k + 1, \ldots, p\},$$
where all $\varrho_{i}$'s are in $\R[x]$.
 \item[(ii)]
Let $B_1 \subset \Bbb{R}^n$ and $B_2 \subset \Bbb{R}^m$ be semi-algebraic
sets. A mapping $F \colon B_1 \to B_2$ is said to be {\em semi-algebraic} if
its graph
$$\{(x, y) \in B_1 \times B_2 \colon y = F(x)\}$$
is a semi-algebraic subset in $\Bbb{R}^n\times\Bbb{R}^m.$
In particular, if $m=1,$ we call the mapping $F$ a semi-algebraic function.
\end{enumerate}
}\end{definition}

The semi-algebraic sets and functions have many
remarkable properties; see, e.g., \cite{RASS,Bochnak1998,HaHV2017}.


\begin{theorem}[Tarski--Seidenberg Theorem]\label{Tarski Seidenberg Theorem}
The image and inverse image of a semi-algebraic set under a semi-algebraic mapping are semi-algebraic sets.
In particular$,$ the projection of a semi-algebraic set is still a semi-algebraic set.
\end{theorem}

The Curve Selection Lemma at infinity (see \cite{HaHV2017,Milnor1968})  will be frequently used in this paper.

\begin{lemma}[Curve Selection Lemma at infinity]\label{CurveSelectionLemmaatinfinity}
Let $A$ be a semi-algebraic subset of $\mathbb{R}^n,$ and let
$$\varrho:=(\varrho_1, \ldots, \varrho_p): \R^n \rightarrow \R^p$$
be a semi-algebraic mapping.
Assume that there exists a sequence $\{x^k\}$ with $x^k \in A,$ $\lim_{k \rightarrow \infty}\|x^k \| = \infty$
and $\lim_{k \rightarrow \infty} \varrho(x^k) = y \in \overline \R^p,$ where $\overline \R := \R \cup \{\infty\} \cup \{-\infty\}.$
Then there exist a positive real number $\epsilon$ and a smooth semi-algebraic curve
$$\phi \colon (0, \epsilon) \to {\mathbb R}^n$$
such that $\phi(t) \in A$ for all $t \in (0, \epsilon),$ $\lim_{t \rightarrow 0}\| \phi(t)\| = \infty,$ and $\lim_{t \rightarrow 0}\varrho(\phi(t)) = y.$
\end{lemma}

In what follows, we will need the following useful results; see \cite{Dries1996}.

\begin{lemma}[Growth Dichotomy Lemma] \label{GrowthDichotomyLemma}
Let $\varrho:(0, \epsilon) \to \R$ be a semi-algebraic function with $\varrho(t) \not= 0$ for all $t \in (0, \epsilon),$ where $\epsilon$ is a positive real number.
Then there exist constants $c \not= 0$ and $q \in \Bbb{Q}$ such that
$$\varrho(t) = ct^q + o (t),$$where $\lim_{t\to 0}\frac{o(t)}{t}=0$.
\end{lemma}

Let $\varrho, \varsigma: (0, \epsilon) \to \R$ be nonzero functions such that $\lim\limits_{t \to 0^+} \varrho(t) \to \infty$ and $\lim\limits_{t \to 0^+} \varsigma(t) \to \infty,$ where $\epsilon$ is a positive real number.
If $\lim\limits_{t \to 0^+} \frac{\varrho(t)}{\varsigma(t)} = c_0,$ where $c_0$ is a positive constant, then we denote this relation by
$$\varrho(t) \simeq \varsigma(t) \ \textrm{as} \ t \to 0^+.$$
\begin{lemma}\label{jiao0421a}
Let $\varrho:(0, \epsilon) \to \R$ be a continuously differentiable semi-algebraic function with $\varrho(t) \not= 0$ for all $t \in (0, \epsilon),$ where $\epsilon$ is a positive real number$,$ and $\varrho(t) \to +\infty$ as $t \to 0^+.$
Then
\begin{align}\label{jiao0421b}
\varrho(t) \simeq t \varrho'(t) \ \textrm{as} \ t \to 0^+.
\end{align}
\end{lemma}
\begin{proof}
Since $\varrho$ is a semi-algebraic function, by  Lemma \ref{GrowthDichotomyLemma}, we can write
\begin{align*}
\varrho(t) = \ \bar c t^{\bar q} + 0( t),
\end{align*}
for some $\bar c \not= 0$, $\bar q \in \Bbb{Q}$  and  $\lim_{t\to 0}\frac{o(t)}{t}=0$.
Clearly, $\bar q<0$, due to $\varrho(t) \to +\infty$ as $t \to 0^+.$
On the other hand, by the continuous differentiablity of $\varrho$, it yields
\begin{align*}
\varrho'(t) = \ \bar c \bar q  t^{\bar q - 1} + \textrm{ higher order terms in } t.
\end{align*}
This shows \eqref{jiao0421b} as $t \to 0^+.$
\end{proof}

\begin{lemma}[Monotonicity Lemma] \label{MonotonicityLemma}
Let $a < b$ in $\mathbb{R}.$ If $\varrho \colon [a, b] \rightarrow
\mathbb{R}$ is a semi-algebraic function, then there is a partition $a
=: t_1 < \cdots < t_{N} := b$ of $[a, b]$ such that $\varrho|_{(t_l, t_{l +
1})}$ is $C^1,$ and either constant or strictly monotone$,$ for $l \in
\{1, \ldots, N - 1\}.$
\end{lemma}

\section{Tangency Variety and Its Properties}\label{Sec:3}

In this section, we introduce some concepts related to the vector polynomial optimization problem~\eqref{problem}, and study their properties.
\begin{definition}\label{tangency-variety}{\rm
By {\it tangency variety of $f$ on $S$} we mean the set
\begin{align*}
\Gamma(f,S):=\left\{x\in S\colon
\left\{
\begin{aligned}
&\textrm{there exist } (\tau, \lambda, \nu, \mu) \in (\R^p_+\times \R^l \times \R^m_+\times \R) \setminus \{\bfz\} \ \textrm{such that}\\
&\sum\limits_{k=1}^p\tau_k \nabla f_k(x)- \sum\limits_{i=1}^l\lambda_i\nabla g_i(x)-\sum\limits_{j=1}^m\nu_j\nabla h_j(x)-\mu x= \bfz\\
&\textrm{and } \nu_j h_j(x)=0,\ j=1,\ldots,m
\end{aligned}
		\right.
\right\},
\end{align*}
where $\nabla f_k (x)$ stands for the gradient of $f_k$ at $x.$
}\end{definition}

\begin{lemma}\label{unboundness}
Let $f:\R^n \to \R^p$ be a polynomial mapping and $S$ be defined as \eqref{SAset}$,$ then $\Gamma(f,S)$ is an unbounded nonempty semi-algebraic set.
\end{lemma}
\begin{proof}
Clearly, it follows from Theorem~\ref{Tarski Seidenberg Theorem} that  $\Gamma(f,S)$ is semi-algebraic.

Now, we claim that $\Gamma(f,S) \not= \emptyset.$
Indeed, for given $r > 0$, denote by
 $$\sph_r := \{x \in \R^n \colon \| x \|^2 = r^2\}.$$
 Then $\sph_r$ is nonempty, bounded and closed, thus the intersection $\sph_r\cap S$ is also nonempty and compact for $r$ large enough, and so is the image $f(\sph_r\cap S)$.
Therefore, the optimization problem
$${\rm Min}_{\R^p_+}\{ f(x) \colon x \in \sph_r\cap S\}$$
admits a Pareto solution. Denote the Pareto solution as $x(r) \in \sph_r\cap S.$
The celebrated Fritz-John optimality conditions \cite[Theorem 7.4]{Jahn2004} imply that $x(r) \in \Gamma(f,S),$ and so $\Gamma(f,S) \not= \emptyset.$
Note that if $r \to \infty$ then $\|x(r)\| = r \to \infty,$ then $\Gamma(f,S)$ is  unbounded and we complete the proof.
\end{proof}


In what follows, we need a constraint qualification ``at infinity", which is inspired by \cite[Definition 3.1]{Pham2020arXiv}, to deal with the case when Pareto solutions occur at infinity.
\begin{definition}\label{regulity}{\rm
The constraint set $S$ is said to satisfy the {\it Mangasarian--Fromovitz constraint qualification at infinity} (${\rm (MFCQ)_{\infty}}$ in short), if there exists a real number $R_0 > 0$ such that for each $x \in S, \|x \| \geq R_0,$ the gradient vectors $\nabla g_i(x),$ $i = 1,\ldots, l,$ are linearly independent and there exists a vector $v \in \R^n$ such that
$$\langle \nabla g_i(x), v\rangle = 0,\ i = 1,\ldots, l \ \ \textrm{and} \ \ \langle \nabla h_j(x), v\rangle > 0, \ j \in J(x),$$
where $J(x) := \{j \in \{1, \ldots, m\} \colon h_j(x) = 0\}$ is the set of {\it active constraint indices}.
}\end{definition}

\begin{remark}{\rm
In order to deal with the case when optimal solutions to polynomial optimization problems occur at infinity, another constraint qualification ``at infinity" called regular at infinity was introduced by \cite[Definition 3.3]{Dinh2014}.
Recall that the constraint set $S$ is said to be {\it regular at infinity} if there exists a real number $R_0 > 0$ such that for each $x \in S, \|x \| \geq R_0,$ the gradient vectors $\nabla g_i(x),$ $i = 1,\ldots, l,$ and $\nabla h_j(x),$ $j \in J(x),$ are linearly independent, where
$$J(x) := \{j \in \{1, \ldots, m\} \colon h_j(x) = 0\}$$
is called the set of {\it active constraint indices}.

Observe that ${\rm (MFCQ)_{\infty}}$ of $S$ is weaker than the regularity at infinity, therefore the following results obtained in the paper can also be guaranteed under regularity at infinity.
}\end{remark}

\begin{lemma}\label{jiao0419b}
If the unbounded set $S$ $($defined as in \eqref{SAset}$)$  satisfies ${\rm (MFCQ)_{\infty}},$ then for each $x\in \Gamma(f,S),$ $\|x\| \gg 1,$ there exist real numbers $\tau_k \in \R_+$ with $\sum_{k = 1}^p \tau_k = 1,$ $\lambda_i, \nu_j,$ and $\mu$ such that
\begin{align*}
&\sum\limits_{k=1}^p\tau_k \nabla f_k(x)- \sum\limits_{i=1}^l\lambda_i\nabla g_i(x)-\sum\limits_{j=1}^m\nu_j\nabla h_j(x)-\mu x=\bfz, \textrm{and } \\
&\nu_j h_j(x)=0,\ j=1,\ldots,m.
\end{align*}
\end{lemma}
\begin{proof}
Since $S$ is unbounded, so is $\Gamma(f,S)$ by Lemma~\ref{unboundness}.
Let $x \in \Gamma(f,S).$
It follows from Definition~\ref{tangency-variety} that there exist $\tau_k,\nu_j\in\R_+,~\lambda_i, \mu \in \R,$ not all zero, such that
\begin{align}
&\sum\limits_{k=1}^p\tau_k \nabla f_k(x)- \sum\limits_{i=1}^l\lambda_i\nabla g_i(x)-\sum\limits_{j=1}^m\nu_j\nabla h_j(x)-\mu x=\bfz, \label{jiao01}\\
& \nu_j h_j(x)=0,\ j=1,\ldots,m. \label{jiao02}
\end{align}

Now, it remains to show, without loss of generality, that $\sum\limits_{k=1}^p\tau_k >0,$ provided that $x \in \Gamma(f,S),$ $\|x \| \gg 1.$
Assume to the contrary that $\sum\limits_{k=1}^p\tau_k = 0,$ then it follows from \eqref{jiao01} and \eqref{jiao02} that
\begin{align*}
&\sum\limits_{i=1}^l\lambda_i\nabla g_i(x)+\sum\limits_{j=1}^m\nu_j\nabla h_j(x)+\mu x=\bfz,\\
& \nu_j h_j(x)=0,\ j=1,\ldots,m,
\end{align*}
for some $\lambda_i, \nu_j, \mu \in \R,$ not all zero.
By using the Curve Selection Lemma at infinity (Lemma~\ref{CurveSelectionLemmaatinfinity}), there exist a positive real number $\epsilon$, a smooth semi-algebraic curve $\varphi(t)$ and semi-algebraic functions $\lambda_i(t), \nu_j(t), \mu(t), t \in (0,\epsilon],$ such that
\begin{itemize}
\item[(a1)] $\varphi(t) \in S$ for $ t \in (0, \epsilon];$
\item[(a2)] $\|\varphi(t)\| \rightarrow +\infty$ as $t \rightarrow 0^+;$
\item[(a3)] $\sum_{i=1}^l\lambda_i(t)\nabla g_i(\varphi(t))+\sum_{j=1}^m\nu_j(t)\nabla h_j(\varphi(t))+\mu(t) \varphi(t) \equiv \bfz;$ and
\item[(a4)] $\nu_j(t) h_j(\varphi(t))\equiv0,\ j=1,\ldots,m.$
\end{itemize}

Since the functions $\nu_j$ and $h_j \circ \varphi$ [note that here and hereafter we denote $h_j(\varphi(t)):= (h_j \circ \varphi)(t)$ in the variable $t$] are semi-algebraic, it follows from the Monotonicity Lemma (Lemma~\ref{MonotonicityLemma}) that for $\epsilon > 0$ small enough, these functions are either constant or strictly monotone.
Then, by (a4), we can see that either  $\nu_j(t) \equiv 0$ or $(h_j\circ\varphi)(t) \equiv 0;$ in particular,
\begin{align}\label{jiao0415}
\nu_j(t) \frac{d}{dt} (h_j\circ\varphi)(t) \equiv 0, \quad j=1,\ldots,m.
\end{align}
It then follows from (a3) that
\begin{align*}
0\ & =\ \sum_{i=1}^l\lambda_i(t)\left\langle\nabla g_i(\varphi(t)), \frac{d \varphi}{dt} \right\rangle +\sum_{j=1}^m\nu_j(t)\left\langle \nabla h_j(\varphi(t)), \frac{d \varphi}{dt} \right\rangle  + \mu(t)\left\langle \varphi(t), \frac{d \varphi}{dt} \right\rangle \\
& =\ \sum_{i=1}^l\lambda_i(t)\frac{d}{dt}(g_i\circ \varphi) (t) +\sum_{j=1}^m\nu_j(t)\frac{d}{dt}(h_j\circ \varphi) (t) + \frac{\mu(t)}{2}\frac{d \|\varphi(t)\|^2}{dt} \\
& =\ \frac{\mu(t)}{2}\frac{d \|\varphi(t)\|^2}{dt}. \tag{by \eqref{jiao0415} and (a1)}
\end{align*}
Therefore $\mu(t) \equiv 0$ by (a2), which implies
\begin{equation}\label{new1}
\sum\limits_{i=1}^l\lambda_i(t) \nabla g_i(\varphi(t))+\sum\limits_{j\in J(\varphi(t))}\nu_j(t) \nabla h_j(\varphi(t))=\bfz.
\end{equation}
By ${\rm (MFCQ)_{\infty}}$, there exists $v\in\R^n$ such that
$$\langle \nabla g_i(\varphi(t)), v\rangle = 0,\ i = 1,\ldots, l \ \ \textrm{and} \ \ \langle \nabla h_j(\varphi(t)), v\rangle > 0, \ j \in J(x).$$
This, combined with \eqref{new1}, yields
$$\sum\limits_{j\in J(\varphi(t))}\langle\nu_j(t)\nabla h_i(\varphi(t)),v\rangle=0.$$
Thus $\nu_j(t)=0$, for all $j\in J(\varphi(t))$. Then by \eqref{new1}, $\lambda_i(t),~i=1,\ldots,l$ are not all zero and
$$\sum\limits_{i=1}^l\lambda_i(t) \nabla g_i(\varphi(t))=\bfz.$$
which contradicts the linear independence of $\nabla g_i(\varphi(t)),~i=1,\ldots,l.$
Hence, $\sum_{k=1}^p\tau_k >0,$ and without loss of generality, we may get $\sum_{k=1}^p\tau_k = 1$ by normalization.
\end{proof}

\section{Palais--Smale Condition, Cerami Condition, $M$-tameness and Properness}\label{Sec:4}
Recall the unbounded semi-algebraic set $S$ defined as \eqref{SAset}
introduced in the Section \ref{Sec:1}.
Given a restrictive polynomial mapping $f:=(f_1, \ldots, f_p): S \to \R^p$ and a value $\bar y \in  \ER^p.$
First, we define the (extended) {\it Rabier function} $v \colon \R^n \to \ER$ by
\begin{align}\label{Rabierfunction}
v(x):= \inf\left\{\left\|\sum_{k = 1}^p \tau_k \nabla f_k(x) - \sum_{i = 1}^l \lambda_i\nabla g_i(x) - \sum_{j = 1}^m \nu_j\nabla h_j(x)\right\| \colon
\left\{
\begin{aligned}
&\tau_k \geq 0 \ \textrm{with} \ \sum_{k = 1}^p \tau_k = 1, \\
&(\lambda, \nu) \in \R^l \times \R^m_+,\ \textrm{and }\\
&\nu_j h_j(x)=0, j=1,\ldots,m
\end{aligned}
		\right.
\right\}.
\end{align}
Next, we consider the following sets:
\begin{align*}
\widetilde{K}_{\infty, \leq \bar y}(f, S)&:=\left\{y\in \R^p \colon
\left\{
\begin{aligned}
&\exists\ \{x^{\ell}\} \subset S \ \textrm{with} \ f(x^{\ell}) \leq \bar y\ \textrm{and}  \ \|x^{\ell} \| \to \infty \\
&\textrm{such that}  \ f(x^{\ell}) \to y,\  v(x^{\ell}) \to 0 \ \textrm{as}\ {\ell} \to \infty
\end{aligned}
		\right.
\right\}, \\
{K}_{\infty, \leq \bar y}(f, S)&:= \left\{y\in \R^p \colon
\left\{
\begin{aligned}
&\exists\ \{x^{\ell}\} \subset S \ \textrm{with} \ f(x^{\ell}) \leq \bar y\ \textrm{and}  \ \|x^{\ell}\| \to \infty \\
&\textrm{such that}  \ f(x^{\ell}) \to y,\  \|x^{\ell}\|\ v(x^{\ell}) \to 0 \ \textrm{as}\ {\ell}\to \infty
\end{aligned}
		\right.
\right\}, \\
{T}_{\infty, \leq \bar y}(f, S)&:= \left\{y\in \R^p \colon
\left\{
\begin{aligned}
&\exists\ \{x^{\ell}\} \subset \Gamma(f, S) \ \textrm{with} \ f(x^{\ell}) \leq \bar y\ \textrm{and}  \ \|x^{\ell} \| \to \infty \\
&\textrm{such that}  \ f(x^{\ell}) \to y \ \textrm{as}\ {\ell} \to \infty
\end{aligned}
		\right.
\right\}.
\end{align*}

If $\bar y = (+\infty, \ldots, +\infty),$ the notations $\widetilde{K}_{\infty, \leq \bar y}(f, S),$ ${K}_{\infty, \leq \bar y}(f, S)$ and ${T}_{\infty, \leq \bar y}(f, S)$ will be written as $\widetilde{K}_{\infty}(f, S),$ ${K}_{\infty}(f, S)$ and ${T}_{\infty}(f, S),$ respectively.
We would note here that all of the sets mentioned above can be computed effectively as shown recently in \cite{Dias2015,Dias2017,Dias2021,Jelonek2014}.

The following result is the constrictive version of \cite[Proposition 3.2]{Kim2019}, while as shown below, the  ${\rm (MFCQ)_{\infty}}$ of $S$ plays an essential role.
\begin{theorem}\label{relationship1}
Let $S$ be defined as in \eqref{SAset}$,$ $f: S \to \R^p$ be a restrictive polymonial mapping and $\bar y \in \ER^p.$
Then the following inclusion holds$,$
\begin{align}\label{jiao0425a}
{K}_{\infty, \leq \bar y}(f, S) \subset \widetilde{K}_{\infty, \leq \bar y}(f, S).
\end{align}
If in addition the set $S$ satisfies ${\rm (MFCQ)_{\infty}}$, then
\begin{align}\label{jiao0425b}
{T}_{\infty, \leq \bar y}(f, S) \subset {K}_{\infty, \leq \bar y}(f, S).
\end{align}
\end{theorem}
\begin{proof}
By definition, the inclusion \eqref{jiao0425a} is satisfied immediately.

Now, we show the inclusion \eqref{jiao0425b} under ${\rm (MFCQ)_{\infty}}$.

Taking any $y \in {T}_{\infty, \leq \bar y}(f, S),$ [if ${T}_{\infty, \leq \bar y}(f, S) = \emptyset,$ then the inclusion \eqref{jiao0425b} holds trivially], by definition there exist sequences $\{x^{\ell}\} \subset S$ and $\{(\tau^\ell, \lambda^\ell, \nu^\ell, \mu^\ell)\} \subset (\R^p_+ \times \R^l \times \R^m_+ \times \R) \setminus \{\bfz\},$ such that
\begin{itemize}
\item[(b1)] $\lim_{\ell \to \infty}\|x^\ell \| = + \infty;$
\item[(b2)] $\lim_{\ell \to \infty} f(x^\ell) = y;$
\item[(b3)] $f(x^\ell) \leq \bar y;$
\item[(b4)] $\sum_{k=1}^p\tau_k^\ell \nabla f_k(x^\ell)- \sum_{i=1}^l\lambda_i^\ell \nabla g_i(x^\ell)-\sum_{j=1}^m\nu_j^\ell \nabla h_j(x^\ell)-\mu^\ell x^\ell=\bfz;$ and
\item[(b5)] $\nu_j^\ell h_j(x^\ell) = 0,\ j=1,\ldots,m.$
\end{itemize}
Without loss of generality, for each $\ell \in \N,$ we can normalize the vector $(\tau^\ell, \lambda^\ell, \nu^\ell, \mu^\ell)$ by
$$\|(\tau^\ell, \lambda^\ell, \nu^\ell, \mu^\ell)\| = 1.$$
Let
\begin{align*}
\mathcal{U}:=\left\{(x, \tau, \lambda, \nu, \mu) \in \mathcal{V} \colon
\left\{
\begin{aligned}
&\sum_{k=1}^p\tau_k \nabla f_k(x)- \sum_{i=1}^l\lambda_i \nabla g_i(x)-\sum_{j=1}^m\nu_j \nabla h_j(x)-\mu x=\bfz\\
&f(x) \leq \bar y, \ \|(\tau, \lambda, \nu, \mu)\| = 1, \ \nu_j h_j(x) = 0,\ j=1,\ldots,m
\end{aligned}
		\right.
\right\},
\end{align*}
where $\mathcal{V} = S \times \R^p \times \R^l \times \R^m \times \R.$
Observe that, $\mathcal{U}$ is a semi-algebraic set in $\R^{n+p+l+m+1}$ and the sequence $\{(x^\ell, \tau^\ell, \lambda^\ell, \nu^\ell, \mu^\ell)\} \subset \mathcal{U}$ tends to infinity in the sense that $\| (x^\ell, \tau^\ell, \lambda^\ell, \nu^\ell, \mu^\ell) \| \to \infty$ as $\ell \to \infty.$
Now, by using the Curve Selection Lemma at infinity (Lemma~\ref{CurveSelectionLemmaatinfinity}) for the semi-algebraic mapping
$$\mathcal{U} \to \R^p, \ (x, \tau, \lambda, \nu, \mu) \mapsto f(x),$$
there exist a positive real number $\epsilon$ and a smooth semi-algebraic curve
\begin{align*}
(\varphi, \tau, \lambda, \nu, \mu): (0, \epsilon) &\to \R^n \times \R^p_+ \times \R^l \times \R^m_+ \times \R \\
t &\mapsto \left(\varphi(t), \tau(t), \lambda(t), \nu(t), \mu(t)\right)
\end{align*}
such that
\begin{itemize}
\item[(c1)] $\lim_{t \to 0^+}\|\varphi(t)\| \rightarrow +\infty;$
\item[(c2)] $\lim_{t \to 0^+}f(\varphi(t)) = y;$
\end{itemize}
and for $ t \in (0, \epsilon),$
\begin{itemize}
\item[(c3)] $\varphi(t) \in S$ and $f(\varphi(t))\leq \bar y;$
\item[(c4)] $\sum_{k=1}^p\tau_k(t)\nabla f_k(\varphi(t))-\sum_{i=1}^l\lambda_i(t)\nabla g_i(\varphi(t))-\sum_{j=1}^m\nu_j(t)\nabla h_j(\varphi(t))-\mu(t) \varphi(t) \equiv \bfz;$
\item[(c5)] $\nu_j(t) h_j(\varphi(t))\equiv0,\ j=1,\ldots,m;$ and
\item[(c6)] $\|(\tau(t), \lambda(t), \nu(t), \mu(t))\|\equiv 1.$
\end{itemize}

Because  $\tau_k,$ $\lambda_i,$ $\nu_j,$ $\mu,$ and $f_k \circ \varphi$ are semi-algebraic, it follows from the Monotonicity Lemma (Lemma~\ref{MonotonicityLemma}) again that for $\epsilon > 0$ small enough, these functions are either constant or strictly monotone.
Then, by (c5),   either  $\nu_j(t) \equiv 0$ or $(h_j\circ\varphi)(t) \equiv 0$.  Consequently,
\begin{align}\label{jiao0419a}
\nu_j(t) \frac{d}{dt} (h_j\circ\varphi)(t) \equiv 0, \quad j=1,\ldots,m.
\end{align}
Now, by  (c4) we obtain
\begin{align*}
&\frac{1}{2}\mu(t)\frac{d (\|\varphi(t)\|^2)}{dt}\\
& = \ \mu(t) \langle \varphi(t),  \varphi'(t)\rangle \\
& =\ \sum_{k=1}^p\tau_k(t) \left\langle \nabla f_k(\varphi(t)), \varphi'(t) \right\rangle - \sum_{i=1}^l\lambda_i(t) \left\langle \nabla g_i(\varphi(t)), \varphi'(t)\right\rangle - \sum_{j=1}^m\nu_j(t) \left\langle \nabla h_j(\varphi(t)), \varphi'(t) \right\rangle \\
& =\  \sum_{k=1}^p\tau_k(t) \frac{d}{dt}(f_k \circ \varphi)(t) - \sum_{i=1}^l\lambda_i(t) \frac{d}{dt}(g_i \circ \varphi)(t) - \sum_{j=1}^m\nu_j(t)\frac{d}{dt}(h_j \circ \varphi)(t) \\
& =\  \sum_{k=1}^p\tau_k(t) \frac{d}{dt}(f_k \circ \varphi)(t). \tag{by \eqref{jiao0419a} and (c3)}
\end{align*}

Let $P:=\{k\in \{1, \ldots, p\}\colon \tau_k(t) \frac{d}{dt}(f_k\circ \varphi)(t) \not\equiv 0\}.$ Then
\begin{align}\label{objective-complementary}
\frac{\mu(t)}{2}\frac{d \|\varphi(t)\|^2}{dt}  =  \sum_{k \in P}\tau_k(t) \frac{d}{dt}(f_k \circ \varphi)(t).
\end{align}

\begin{itemize}
\item[{\rm Case 1.}] $P = \emptyset.$ Clearly, combining (c1) and \eqref{objective-complementary} implies that $\mu(t) \equiv 0,$ and along with (c4) and (c5), we have
\begin{align}\label{xiangguan}
&\sum_{k=1}^p\tau_k(t)\nabla f_k(\varphi(t))-\sum_{i=1}^l\lambda_i(t)\nabla g_i(\varphi(t))-\sum_{j=1}^m\nu_j(t)\nabla h_j(\varphi(t)) \equiv \bfz,\\
&\nu_j(t) h_j(\varphi(t))\equiv0,\ j=1,\ldots,m.
\end{align}
We claim that $\sum\limits_{k=1}^p\tau_k(t)>0$. Otherwise, $\tau_k(t)=0$ for any $k=1,\ldots,p$. This, combined with \eqref{xiangguan}, yields
\begin{equation}\label{new2}
\sum_{i=1}^l\lambda_i(t)\nabla g_i(\varphi(t))+\sum_{j=1}^m\nu_j(t)\nabla h_j(\varphi(t)) \equiv \bfz,
\end{equation}
By ${\rm (MFCQ)_{\infty}}$, there exists $v\in\R^n$ such that
$$\langle \nabla g_i(\varphi(t)), v\rangle = 0,\ i = 1,\ldots, l \ \ \textrm{and} \ \ \langle \nabla h_j(\varphi(t)), v\rangle > 0, \ j \in J(x).$$
This, combined with \eqref{new2}, yields
$$\sum\limits_{j\in J(\varphi(t))}\langle\nu_j(t)\nabla h_i(\varphi(t)),v\rangle=0.$$
Thus $\nu_j(t)=0$, for all $j\in J(\varphi(t))$. Then by \eqref{new2}, $\lambda_i(t),~i=1,\ldots,l$ are not all zero and
$$\sum\limits_{i=1}^l\lambda_i(t) \nabla g_i(\varphi(t))=\bfz.$$
which contradicts the linear independence of $\nabla g_i(\varphi(t)),~i=1,\ldots,l.$
Consequently, $v(\varphi(t)) \equiv 0$ by \eqref{Rabierfunction}.
Taking (c1)--(c3) into account yields $y \in K_{\infty, \leq \bar y}(f,S).$
\item[{\rm Case 2.}] $P \not= \emptyset.$ For each $k \in P,$ we have $\tau_k(t) \not \equiv 0$ and $\frac{d}{dt}(f_k \circ \varphi)(t)\not \equiv 0,$ thus $(f_k\circ \varphi)(t) \not \equiv y_k.$
It follows from Lemma \ref{GrowthDichotomyLemma} that
\begin{align*}
\tau_k(t) & = \ a_k t^{\alpha_k} + o( t), \\
(f_k \circ \varphi)(t) & = \ y_k + b_k t^{\beta_k} + o( t),
\end{align*}
where $a_k >0,$ $b_k \not= 0$, $\alpha_k, \beta_k \in \Bbb{Q}$ and   $\lim_{t\to 0}\frac{o(t)}{t}=0$.
It follows from (c6) and (c2), respectively, that $\alpha_k \geq 0$ and $\beta_k > 0.$
Moreover, $\gamma:= \min_{k\in P} (\alpha_k + \beta_k) > 0.$
Clearly, $\gamma>\bar \alpha:=\min_{k\in P}\alpha_k$ and
\begin{equation}\label{sumtau}
\sum\limits_{k=1}^{p}\tau_k=\bar a t^{\bar \alpha}+\textrm{ higher order terms in } t,
\end{equation}
where $\bar a$ is a positive constant.

Now, by (c4) and \eqref{objective-complementary}, we have
\begin{align*}
&\frac{\bigg\|\sum\limits_{k=1}^p\tau_k(t)\nabla f_k(\varphi(t))-\sum\limits_{i=1}^l\lambda_i(t)\nabla g_i(\varphi(t))-\sum\limits_{j=1}^m\nu_j(t)\nabla h_j(\varphi(t))\bigg\|}{2\|\varphi(t)\|} \bigg|\frac{d \|\varphi(t)\|^2}{dt}\bigg|\\
=\ & \bigg|\sum_{k \in P}\tau_k(t) \frac{d}{dt}(f_k \circ \varphi)(t)\bigg|
\end{align*}
Note that by Lemma~\ref{jiao0421a}, we have
\begin{align*}
\|\varphi (t) \|^2 \simeq t\frac{d\|\varphi(t)\|^2}{dt}\ \textrm{as} \ t \to 0^+.
\end{align*}
Hence, 
\begin{align*}
&\|\varphi(t)\| \bigg\|\sum\limits_{k=1}^p\tau_k(t)\nabla f_k(\varphi(t))-\sum\limits_{i=1}^l\lambda_i(t)\nabla g_i(\varphi(t))-\sum\limits_{j=1}^m\nu_j(t)\nabla h_j(\varphi(t))\bigg\|\\
\simeq\ & \frac{\bigg\|\sum\limits_{k=1}^p\tau_k(t)\nabla f_k(\varphi(t))-\sum\limits_{i=1}^l\lambda_i(t)\nabla g_i(\varphi(t))-\sum\limits_{j=1}^m\nu_j(t)\nabla h_j(\varphi(t))\bigg\|}{\|\varphi(t)\|} \bigg|t\frac{d \|\varphi(t)\|^2}{dt}\bigg|.
\end{align*}
Taking (c4) and \eqref{objective-complementary} into account, one has
\begin{align*}
& \frac{\bigg\|\sum\limits_{k=1}^p\tau_k(t)\nabla f_k(\varphi(t))-\sum\limits_{i=1}^l\lambda_i(t)\nabla g_i(\varphi(t))-\sum\limits_{j=1}^m\nu_j(t)\nabla h_j(\varphi(t))\bigg\|}{\|\varphi(t)\|} \bigg|t\frac{d \|\varphi(t)\|^2}{dt}\bigg|\\
=\ & 2t \bigg|\sum_{k \in P}\tau_k(t) \frac{d}{dt}(f_k \circ \varphi)(t)\bigg| \\
=\ & a_0t^{\gamma+1} + \textrm{ higher order terms in } t,
\end{align*}
for some constant $a_0 \geq 0.$
On the other hand, taking
$$
\bar \tau_k(t)=\frac{\tau_k(t)}{\sum\limits_{k=1}^p\tau_k(t)},~~\bar \lambda_i(t)=\frac{\lambda_i(t)}{\sum\limits_{k=1}^p\lambda_k(t)}~~\text{and}~~\bar \nu_j(t)=\frac{\nu_j(t)}{\sum\limits_{k=1}^{p}\tau_k(t)}.
$$
we get $\sum\limits_{k=1}^p\bar\tau_k(t)=1$ and
$$\lim\limits_{t \to 0^+} \|\varphi(t)\| \bigg\|\sum\limits_{k=1}^p\bar \tau_k(t)\nabla f_k(\varphi(t))-\sum\limits_{i=1}^l\bar \lambda_i(t)\nabla g_i(\varphi(t))-\sum\limits_{j=1}^m\bar\nu_j(t)\nabla h_j(\varphi(t))\bigg\| = 0,$$
due to $\gamma>\bar\alpha$ and \eqref{sumtau}. This, along with (c1)--(c3), reaches $y \in K_{\infty, \leq \bar y}(f, S).$
\end{itemize}
Thus, the proof is complete.
\end{proof}

\begin{remark}{\rm
\begin{enumerate}
\item [{\rm (i)}] It is worth noting that the assumption on ${\rm (MFCQ)_{\infty}}$ of $S$ is a generic condition in the sense that it holds in an open dense semi-algebraic set of the entire space of input data (see \cite{Bolte2018,Dinh2014,HaHV2017}).
\item [{\rm (ii)}] The inclusion \eqref{jiao0425b} holds under the ${\rm (MFCQ)_{\infty}}$ of the constraint set $S.$
If $S = \R^n,$ the inclusion \eqref{jiao0425b} still holds (of course without any constraint qualifications) in the polynomial mapping setting (see \cite[Proposition 3.2]{Kim2019}), while it may go awry in more general setting, e.g., $f$ is not a polynomial mapping (see \cite[Example 3.1]{Kim2020}).
\end{enumerate}
}\end{remark}

The following example shows that the assumption on ${\rm (MFCQ)_{\infty}}$ of $S$ plays an essential role, and it cannot be dropped.
In other words, the inclusion \eqref{jiao0425b} in Theorem~\ref{relationship1} does {\em not} hold if $S$ does not satisfy ${\rm (MFCQ)_{\infty}}$.
\begin{example}\label{example0717}{\rm
Let $x:=(x_1, x_2, x_3) \in \R^3.$
Let
\begin{align*}
f(x):=&\ (f_1(x), f_2(x)) = \left(x_2x_3, x_1x_3\right), \\
g_1(x):=&\ (1 - x_1x_2x_3)^2 + x_1^2 + x_2^2 - 1, \\
g_2(x):=&\ x_1x_2,\\
h(x):=&\ x_1^3.
\end{align*}
Consider the following vector polynomial optimization problem with constraints
\begin{align}\label{example02}
{\rm Min}_{\mathbb{R}^2_+}\;\big\{f(x)\,\colon \,x\in S\big\},\tag{VPO$_1$}
\end{align}
where $S:=\{x \in \R^3 \colon g_1(x) = 0, g_2(x) = 0, h(x) \geq 0\} = \{(0, 0, x_3) \colon x_3 \in \R\}.$
A simple calculation yields that
$$\nabla f_1 = \left(
  \begin{array}{c} 0\\ x_3\\ x_2 \\
  \end{array}
\right),
\nabla f_2 = \left(
  \begin{array}{c} x_3\\ 0\\ x_1 \\
  \end{array}
\right),$$\\
$$\text{and}~~\nabla g_1 = \left(
  \begin{array}{c} -2(1 - x_1x_2x_3)x_2x_3 + 2x_1\\ -2(1 - x_1x_2x_3)x_1x_3 + 2x_2\\ -2(1 - x_1x_2x_3)x_1x_2\\
  \end{array}
\right),
\nabla g_2 = \left(
  \begin{array}{c} x_2\\ x_1\\ 0 \\
  \end{array}
\right),
\nabla h = \left(
  \begin{array}{c} 3x_1^2\\ 0\\ 0 \\
  \end{array}
\right).$$

By Definition~\ref{tangency-variety}, one has
\begin{align*}
\Gamma(f,S):=&\ \left\{x\in S\colon \left\{
\begin{aligned}&\sum_{k = 1}^2\tau_k \nabla f_k(x) - \sum_{i = 1}^2\lambda_i \nabla g_i (x)-\nu \nabla h (x) - \mu x = \bfz, \\
 &\textrm{for some}\ (\tau_1, \tau_2, \lambda_1, \lambda_2, \nu, \mu) \not= 0   \end{aligned}\right.
\right\}\\
=&\ \{(0, 0, x_3) \colon x_3 \in \R\}.
\end{align*}

Now, we will show that the inclusion \eqref{jiao0425b} fails to hold in the case for problem~\eqref{example02}, for convenience, let $\bar y = (+\infty, +\infty);$ in other words,
\begin{align}\label{jiao-a1}
{T}_{\infty}(f, S) \not\subset {K}_{\infty}(f, S).
\end{align}
Indeed, by calculation, we have
\begin{align*}
{T}_{\infty}(f, S) = \left\{y\in \R^2 \colon
\left\{
\begin{aligned}
&\exists\ \{x^{\ell}\} \subset \Gamma(f, S) \ \textrm{with} \ \|x^{\ell} \| \to \infty \\
&\textrm{such that}  \ f(x^{\ell}) \to y \ \textrm{as}\ {\ell} \to \infty
\end{aligned}
		\right.
\right\} = \{(0,0)\}.
\end{align*}
On the other hand, by \eqref{Rabierfunction}
\begin{align*}
v(x):= \inf\left\{\left\|\sum_{k = 1}^2 \tau_k \nabla f_k(x) - \sum_{i = 1}^2 \lambda_i\nabla g_i(x) - \nu\nabla h (x) \right\| \colon
\left\{
\begin{aligned}
&\tau_1, \tau_2 \geq 0 \ \textrm{with} \ \tau_1 + \tau_2 = 1 \\
&\lambda_1, \lambda_2 \in \R, \nu \in \R_+
\end{aligned}
		\right.
\right\}.
\end{align*}
Now, consider the set $\widetilde{K}_{\infty}(f, S).$ Note that $S= \{(0, 0, x_3) \colon x_3 \in \R\}$, for $x\in S$,
\begin{align}\label{Rabierfunction1}
v(x)=\inf\left\{ \sqrt{(\tau_1^2+\tau_2^2)x_3^2}\colon  \tau_1, \tau_2\in \R_+, \ \tau_1 + \tau_2 = 1\right\} =  \tfrac{\sqrt{2}}{2}|x_3|.\end{align}
Hence,
\begin{align}\label{jiao-a2}
\widetilde{K}_{\infty}(f, S)&:=\left\{y\in \R^2 \colon
\left\{
\begin{aligned}
&\exists\ \{x^{\ell}\} \subset S \ \textrm{with} \ \|x^{\ell} \| \to \infty\ \textrm{such that}  \\
&f(x^{\ell}) \to y,\  v(x^{\ell}) \to 0 \ \textrm{as}\ {\ell} \to \infty
\end{aligned}
		\right.
\right\} = \emptyset,
\end{align}
this is because there is no $\{x^{\ell}\} \subset S$ with $\|x^{\ell} \| \to \infty$ such that $v(x^{\ell}) \to 0$ as ${\ell} \to \infty$ by \eqref{Rabierfunction1}.
Thus, \eqref{jiao-a2} along with \eqref{jiao0425a} implies that ${K}_{\infty}(f, S) = \emptyset.$
As a result, we get \eqref{jiao-a1}.
The reason is that the constraint set $S$ in consideration does not satisfy ${\rm (MFCQ)_{\infty}}$ (the reader may check it by definition easily).
\qed
}\end{example}

\begin{definition}{\rm(\cite[Definition 3.2]{Kim2020})
\begin{itemize}
\item[{\rm(i)}] the restrictive polynomial mapping $f$ on $S$ is called {\it proper at   $\bar y \in \ER^p$} if $$\forall \{x^{\ell}\} \subset S,\ \|x^{\ell}\| \to \infty,\ f(x^{\ell}) \leq \bar y \Longrightarrow \|f(x^{\ell})\| \to \infty\ \textrm{ as }\ \ell \to \infty;$$
\item[{\rm(ii)}] the restrictive polynomial mapping $f$ on $S$ is called {\it proper} if it is proper at every   $\bar y \in \ER^p.$
\end{itemize}
}\end{definition}

\begin{remark}
{\rm 
\begin{enumerate}
\item [{\rm (i)}] In case of $p=1,$ the properness of $f$ on $S$ is weaker than the coercivity of $f$ on $S.$
Remember that $f$ is said to be {\it coercive} on $S$ (see \cite{Bajbar2015,Bajbar2019} for more information in polynomial setting) if
$$\lim_{x \in S,\ \|x\| \rightarrow \infty} f(x) = + \infty.$$
Indeed, if $f$ is coercive on $S,$ then $f$ on $S$ is proper at $\bar y = +\infty.$
Conversely, it may fail to hold in general.
For example, let $f(x) = x$ and $S = \R,$ it is clear that $f$ is proper (at every $\bar y\in\ER$) but not coercive.
\item [{\rm (ii)}] If $p \geq 2,$ the properness of $f$ on $S$ is also weaker than other coercivity conditions, such as {\it $\R^p_+$-zero-coercivity} $f$ on $S$ introduced by Guti\'errez et al \cite[Definition 3.1]{Gutierrez2014}.
	Remember that $f$ is said to be {\it $\R^p_+$-zero-coercive} on $S$ with respect to $\xi \in \R^p_+\setminus \{\bfz\}$ if
$$\lim_{x \in S,\ \|x\| \rightarrow \infty} \langle \xi, f(x) \rangle = + \infty.$$
Actually, if $f$ is $\R^p_+$-zero-coercive on $S,$ then $f$ on $S$ is proper at $\bar y = (+\infty, \ldots, + \infty).$
Conversely, it may fail to hold in general.
For example, let $f(x_1, x_2) = (x_1, x_2)$ and $S = \R^2,$ it is clear that $f$ is proper (at every $\bar y\in\ER^p$) but not $\R^p_+$-zero-coercive with respect to any $\xi \in \R^p_+\setminus \{\bfz\}$.
\end{enumerate}
}
\end{remark}

\begin{definition}{\rm Let $f: S \to \R^p$ be a restrictive polynomial mapping and $\bar y \in  \ER^p$.
\begin{itemize}
\item[{\rm(i)}] The restrictive polynomial mapping $f$ on $S$ satisfies the {\it Palais--Smale condition} at  $\bar y$ if $\widetilde{K}_{\infty, \leq \bar y}(f, S) = \emptyset;$
\item[{\rm(ii)}] The restrictive polynomial mapping $f$ on $S$ satisfies the {\it Cerami condition} (or {\it weak Palais--Smale condition}) at $\bar y$ if ${K}_{\infty, \leq \bar y}(f, S) = \emptyset;$
\item[{\rm(iii)}] The restrictive polynomial mapping $f$ on $S$ is called {\it M-tame} at $\bar y$ if $T_{\infty, \leq \bar y}(f, S) = \emptyset.$
\end{itemize}
}\end{definition}

Observe that if the restrictive polynomial mapping $f$ on $S$ is proper at $\bar y \in \ER^p,$ then by
definition,
$$T_{\infty, \leq \bar y}(f, S) = {K}_{\infty, \leq \bar y}(f, S) = \widetilde{K}_{\infty, \leq \bar y}(f, S) =  \emptyset.$$
But  not vice versa.  (see \cite{Kim2020}).

\begin{theorem}\label{equivalent1}
Let $S$ be defined as in \eqref{SAset} and presume that the set $S$ satisfies ${\rm (MFCQ)_{\infty}}$.
Let $f: S \to \R^p$ be a restrictive polynomial mapping.
Presume that there exists $\bar y \in f(S)$ such that the section $[f(S)]_{\bar y}$ is bounded.
Then the following assertions are equivalent$:$
\begin{itemize}
\item[{\rm(i)}] The restrictive polynomial mapping $f$ on $S$ is proper at   $\bar y.$
\item[{\rm(ii)}] The restrictive polynomial mapping $f$ on $S$ satisfies the Palais--Smale condition at   $\bar y.$
\item[{\rm(iii)}] The restrictive polynomial mapping $f$ on $S$ satisfies the Cerami condition at  $\bar y.$
\item[{\rm(iv)}] The restrictive polynomial mapping $f$ on $S$ is $M$-tame at   $\bar y.$
\end{itemize}
If in addition$,$ one of the above  conditions holds$,$ then the section $[f(S)]_{\bar y}$ is compact.
\end{theorem}
\begin{proof}
By definition, the implications $[{\rm (i)}\Rightarrow {\rm (ii)} \Rightarrow {\rm (iii)}]$ is satisfied immediately.
Since the set $S$ satisfies ${\rm (MFCQ)_{\infty}}$, thus the implication
$[{\rm (iii)} \Rightarrow {\rm (iv)}]$ follow from Theorem \ref{relationship1}.

Now, we will show the implication $[{\rm (iv)} \Rightarrow {\rm (i)}]$. Assume to the contrary that the restrictive polynomial mapping $f$ on $S$ is {\it not} proper at   $\bar y.$
Then by definition, we have
$$\exists \{x^{\ell}\} \subset S,\ \|x^{\ell}\| \to \infty,\ f(x^{\ell}) \leq \bar y \Longrightarrow \|f(x^{\ell})\| \to M\ \textrm{ as }\ \ell \to \infty,$$
where $M$ is a nonnegative constant.

For each fixed $\ell \in \N,$ we consider the problem
\begin{align}\label{jiao0416a}
{\rm Min}_{\mathbb{R}^p_+}\;\left\{f(x)\,\colon \,x\in S,\ f(x) \leq \bar y\ \textrm{and}\ \|x \|^2 = \| x^\ell\|^2\right\}.\tag{P}
\end{align}
Since the set $\{ x \in \R^n \, \colon \, x\in S,\ f(x) \leq \bar y\ \textrm{and}\ \|x \|^2 = \| x^\ell\|^2\}$ is nonempty and compact, and $f$ is continuous, thus the problem~\eqref{jiao0416a} has a Pareto solution, say $z^\ell.$ According to Fritz-John optimality conditions \cite[Theorem 7.4]{Jahn2004}, there are $({\bf a}, {\bf b}, {\bf c}, {\bf d}, {\bf e}) \in (\R^p_+ \times \R^l \times \R^m_+ \times \R^p_+ \times \R) \setminus \{\bfz\}$ such that
\begin{align*}
& \sum_{k = 1}^p {\bf a}_k\nabla f_k(z^\ell) - \sum_{i = 1}^l {\bf b}_i\nabla g_i(z^\ell) - \sum_{j = 1}^m {\bf c}_j\nabla h_j(z^\ell) + \sum_{k = 1}^p {\bf d}_k\nabla f_k(z^\ell) - 2{\bf e} z^\ell = \bfz, \\
& {\bf c}_j h_j(z^\ell) = 0, j = 1, \ldots, m,  \ \textrm{ and }\ {\bf d}_k (f_k(z^\ell) - \bar y_k) = 0, k = 1, \ldots, p.
\end{align*}
Letting $\tau_k:={\bf a}_k + {\bf d}_k,$ $k = 1, \ldots, p,$ $\lambda_i := {\bf b}_i,$ $i = 1, \ldots, l,$ $\nu_j := {\bf c}_j,$ $j = 1, \ldots, m$ and $\mu:=2{\bf e},$ it yields that
\begin{align*}
& \sum_{k = 1}^p \tau_k\nabla f_k(z^\ell) - \sum_{i = 1}^l \lambda_i\nabla g_i(z^\ell) - \sum_{j = 1}^m \nu_j\nabla h_j(z^\ell) - \mu z^\ell = \bfz, \\
& \nu_j h_j(z^\ell) = 0, j = 1, \ldots, m,  \ \textrm{ and }\ {\bf d}_k (f_k(z^\ell) - \bar y_k) = 0, k = 1, \ldots, p.
\end{align*}
Clearly, $(\tau, \lambda, \nu, \mu) \not= \bfz,$ then $z^\ell \in \Gamma (f, S).$

Consequently,     $\{ z^\ell \}$ has the following properties:
\begin{itemize}
\item[(d1)] $\{z^\ell\} \subset \Gamma (f, S);$
\item[(d2)] $\|z^\ell\| = \|x^\ell\| \to +\infty$ as $\ell \to +\infty;$ and
\item[(d3)] $f(z^\ell) \leq \bar y$ for all $\ell \in \N.$
\end{itemize}
Observe that the assumption that $[f(S)]_{\bar y}$ is bounded,  without loss of generality, we assume that $f (z^\ell) \to y.$ Clearly, $y \leq \bar y$.
Thus $y \in {T}_{\infty, \leq \bar y}(f, S).$
That is, ${T}_{\infty, \leq \bar y}(f, S) \not= \emptyset,$ which contradicts to $M$-tameness.

Finally, let us show the compactness of $[f(S)]_{\bar y}$ by condition (i).
Indeed, if the restrictive polynomial mapping $f$ on $S$ is proper at $\bar y,$ then the set $S':=\{x \in S \colon f(x) \leq \bar y\}$ is bounded.
Furthermore, let a sequence $\{x^{\ell}\} \subset S$ satisfies $f(x^\ell) \leq \bar y$ for all $\ell \in \N.$
As  $\{f(x^\ell)\} \subset [f(S)]_{\bar y},$ by the condition (i) we obtain that $\{x^\ell\}$ is bounded.
Meanwhile, by the closedness of $S$ and continuity of the mapping $f,$ it ensures that $S'$ is closed.
Thus, $S'$ is compact, which together with the continuity of the mapping $f,$ yields that the section $[f(S)]_{\bar y}$ is compact.
\end{proof}

\section{Existence of Pareto Solutions}\label{Sec:5}

\begin{theorem}\label{exist-result}
Let $S$ be defined as in \eqref{SAset} and assume that the set $S$ satisfies ${\rm (MFCQ)_{\infty}}$.
Let $f: S \to \R^p$ be a restrictive polynomial mapping.
Presume that there exists $\bar y \in f(S)$ such that the section $[f(S)]_{\bar y}$ is bounded.
Then the problem~\eqref{problem} possesses at least one Pareto solution$,$ if one of the following  conditions holds$:$
\begin{itemize}
\item[{\rm(i)}] The restrictive polynomial mapping $f$ on $S$ is proper at  $\bar y.$
\item[{\rm(ii)}] The restrictive polynomial mapping $f$ on $S$ satisfies the Palais--Smale condition at  $\bar y.$
\item[{\rm(iii)}] The restrictive polynomial mapping $f$ on $S$ satisfies the Cerami condition at   $\bar y.$
\item[{\rm(iv)}] The restrictive polynomial mapping $f$ on $S$ is $M$-tame at $\bar y.$
\end{itemize}
\end{theorem}
\begin{proof}
By Theorem~\ref{equivalent1}, it yields that the section $[f(S)]_{\bar y}$ is compact providing one of the above equivalent conditions (i)--(iv) holds.
Therefore, the result follows by \cite[Theorem 1]{Borwein1983} (or \cite[Theorem 2.10]{Ehrgott2005}).
\end{proof}

\begin{remark}{\rm
Note that if the restrictive polynomial mapping $f$ on $S$ is proper at the   $\bar y \in f(S),$ and the section $[f(S)]_{\bar y}$ is bounded, then the problem~\eqref{problem} obviously possesses at least one Pareto solution.
However, as mentioned in \cite{Kim2020}, the problem of checking a function is proper (or coercive) is {\it strongly NP-hard} even for polynomials of degree 4 (see \cite[Theorem 3.1]{Ahmadi2020}).
}\end{remark}

\begin{example}{\rm
Let $x:=(x_1, x_2) \in \R^2.$
Let
\begin{align*}
f(x):=&\ f(x_1, x_2, x_3) = \left(x_3, (1- x_1x_2)^2 + x_2^2 + x_3^2\right).
\end{align*}
Consider the following vector polynomial optimization problem with constraints
\begin{align}\label{examPareto}
{\rm Min}_{\mathbb{R}^2_+}\;\big\{f(x)\,\colon \,x\in S\big\},\tag{VPO$_2$}
\end{align}
where $S:=\{(x_1, x_2, x_3) \in \R^3 \colon x_1 \geq 0, x_2 \geq 0\}.$
Clearly, the set $S$ satisfies ${\rm (MFCQ)_{\infty}}$.
The image $f(S)$ of $f$ over $S$ can be seen in {\sc Figure}~1.
\begin{figure}[H]
\begin{minipage}{.4\textwidth}
\includegraphics[height=8cm,width=6cm]{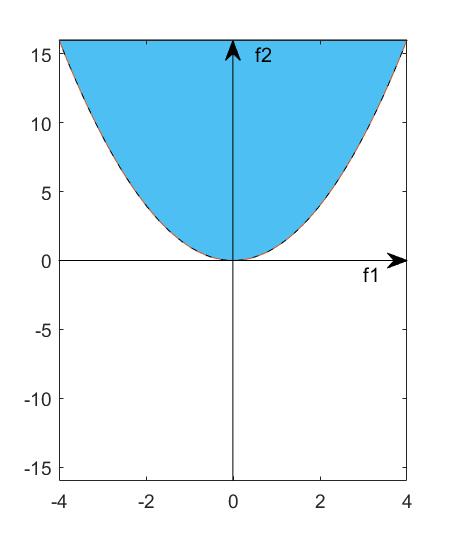}
\caption{The image $f(S)$.}
\end{minipage}
\begin{minipage}{.4\textwidth}
\includegraphics[height=8cm,width=6cm]{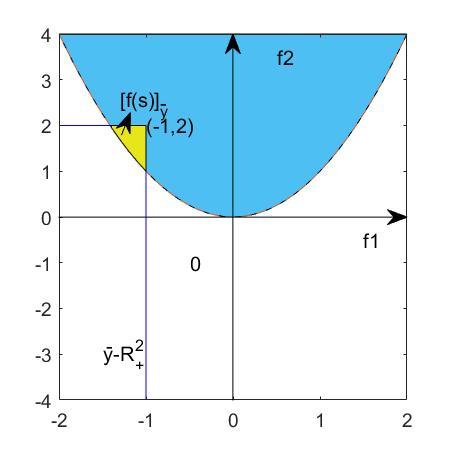}
\caption{The section $[f(S)]_{\bar y}$.}
\end{minipage}
\end{figure}
Let $\bar y = (-1, 2).$
It is clear that the section $[f(S)]_{\bar y}$ is bounded but not closed (see {\sc Figure}~2).
On the other hand, by Definition~\ref{Pareto}, one can easily verify that $\mathrm{sol}\,({\rm VPO}_2) \not= \emptyset.$
In this case, $\widetilde{K}_{\infty, \leq \bar y}(f, S) \not= \emptyset.$
\qed
}\end{example}

\begin{example}{\rm
Let $x:=(x_1, x_2) \in \R^2.$
Let
\begin{align*}
f(x):=&\ (f_1(x), f_2(x)) = \left(x_1^2x_2^4 + x_1^4x_2^2 - 3x_1^2x_2^2 + 1, (x_1 - 1)^2 + (x_2 - 1)^2\right),
\end{align*}
where $f_1(x)$ is known as Motzkin polynomial (see \cite{HaHV2017}).
Consider the following vector polynomial optimization problem with constraints
\begin{align}\label{exam09}
{\rm Min}_{\mathbb{R}^2_+}\;\big\{f(x)\,\colon \,x\in S\big\},\tag{VPO$_3$}
\end{align}
where $S:=\left\{x \in \R^2 \colon h_1(x) = x_1 \geq 0, h_2(x) = x_2 \geq 0\right\} = \R_+ \times \R_+.$
Clearly, the set $S$ satisfies ${\rm (MFCQ)_{\infty}}$, and by definition, we can easily verify that $T_{\infty, \leq \bar y}(f, S) = \emptyset$ for $\bar y = (0, 0).$
Hence, along with Theorem~\ref{exist-result}, $\mathrm{sol}\,({\rm VPO}_3) \not= \emptyset.$

On the other hand, a simple calculation yields that the solution set is $\{(1,1)\}.$
\qed
}\end{example}

\begin{corollary}
Let $S:=\{x \in \R^n \colon Ax = b\},$ where $A\in \R^{m \times n}, b \in \R^m,$ and let $f: \R^n \to \R^p$ be a linear mapping with $f(x):= Cx,$ where $C\in \R^{p \times n}.$
Presume that both the rows of $A$ and $C$ are linearly independent.
If there exists $\bar y \in f(S)$ such that the section $[f(S)]_{\bar y}$ is bounded.
Then the problem~\eqref{problem} possesses at least one Pareto solution.
\end{corollary}

\section{Conclusions and Further Discussions}\label{Sec:7}
In this paper, we derive some sufficient conditions for the existence of Pareto solutions to the consitained vector polynomial optimization problem~\eqref{problem}.
Such sufficient conditions are based on the Palais--Smale condition, the Cerami condition, the $M$-tameness, and the properness for the restrictive polynomial mapping $f$ over $S$.
Among others, it is worth mentioning that our results are derived under the assumption ${\rm (MFCQ)_{\infty}}$ of $S$, which are significantly different to the results in the literature \cite{Kim2019,Kim2020}.
Now, having the results in hand, we will close this paper by mentioning the following possible research directions as future investigations.
\begin{itemize}
\item[{\rm (i)}]
[Finding (weak) Pareto solutions] When all functions in the problem~\ref{problem} are linear, Blanco et al \cite{Blanco2014} obtain the set of Pareto solutions by using a semidefinite programming method.
When the functions in the problem~\ref{problem} are convex polynomials, Moment-SOS relaxation methods are invoked to compute Pareto solutions in \cite{Jiao2020,Lee2018,Lee2019,Lee2021}. Further important results on computing Pareto solutions/values to the problem~\ref{problem} are given in \cite{Magron2014,Magron2015,Nie2021}.
A natural question arises: how to compute the Pareto solutions/values to the problem~\ref{problem} without any convexity assumptions?
\item[{\rm (ii)}]  [On vector polynomial variational inequality problems] By using the similar techniques (with possibly significantly modifications), it is also very interesting to investigate  the  vector polynomial variational inequality problems, such problems can be seen in \cite{Huong2016,Yen2016}.
\end{itemize}

\subsection*{Acknowledgments}
The authors wish to thank Ti\ees n-S\ow n Ph\d{a}m for many valuable suggestions.
This work was supported by the National Natural Sciences Foundation of China (11971339, 11771319).


\begin{thebibliography}{10}

\bibitem{Ahmadi2020}
A.~A. Ahmadi and J.~Zhang.
\newblock On the complexity of testing attainment of the optimal value in
  nonlinear optimization.
\newblock {\em Mathematical Programming}, 184(1-2):221--241, 2020.

\bibitem{Bajbar2015}
T.~Bajbar and O.~Stein.
\newblock Coercive polynomials and their {N}ewton polytopes.
\newblock {\em SIAM Journal on Optimization}, 25(3):1542--1570, 2015.

\bibitem{Bajbar2019}
T.~Bajbar and O.~Stein.
\newblock Coercive polynomials: stability, order of growth, and {N}ewton
  polytopes.
\newblock {\em Optimization}, 68(1):99--124, 2019.

\bibitem{Bao2007}
T.~Q. Bao and B.~S. Mordukhovich.
\newblock Variational principles for set-valued mappings with applications to
  multiobjective optimization.
\newblock {\em Control} \& {\em Cybernetics}, 36(3):531--562, 2007.

\bibitem{Bao2010MP}
T.~Q. Bao and B.~S. Mordukhovich.
\newblock Relative {P}areto minimizers for multiobjective problems: existence
  and optimality conditions.
\newblock {\em Mathematical Programming}, 122(2):301--347, 2010.

\bibitem{Belousov2002}
E.~G. Belousov and D.~Klatte.
\newblock A {F}rank--{W}olfe type theorem for convex polynomial programs.
\newblock {\em Computational Optimization and Applications}, 22(1):37--48,
  2002.

\bibitem{RASS}
R.~Benedetti and J.~Risler.
\newblock {\em Real Algebraic and Semi-algebraic Sets}.
\newblock Hermann, Paris, 1991.

\bibitem{Blanco2014}
V.~Blanco, J.~Puerto, and S.~E. H.~B. Ali.
\newblock A semidefinite programming approach for solving multiobjective linear
  programming.
\newblock {\em Journal of Global Optimization}, 58(3):465--480, 2014.

\bibitem{Bochnak1998}
J.~Bochnak, M.~Coste, and M.-F. Roy.
\newblock {\em Real Algebraic Geometry}, volume~36.
\newblock Springer-Verlag, New York, 1998.

\bibitem{Bolte2018}
J.~Bolte, A.~Hochart, and E.~Pauwels.
\newblock Qualification conditions in semialgebraic programming.
\newblock {\em SIAM Journal on Optimization}, 28(2):1867--1891, 2018.

\bibitem{Borwein1983}
J.~M. Borwein.
\newblock On the existence of {P}areto efficient points.
\newblock {\em Mathematics of Operations Research}, 8(1):64--73, 1983.

\bibitem{Corley1980}
H.~Corley.
\newblock An existence result for maximization with respect to cones.
\newblock {\em Journal of Optimization Theory and Applications},
  31(2):277--281, 1980.

\bibitem{Deng1998a}
S.~Deng.
\newblock Characterizations of the nonemptiness and compactness of solution
  sets in convex vector optimization.
\newblock {\em Journal of Optimization Theory and Applications},
  96(1):123--131, 1998.

\bibitem{Deng1998b}
S.~Deng.
\newblock On efficient solutions in vector optimization.
\newblock {\em Journal of Optimization Theory and Applications},
  96(1):201--209, 1998.

\bibitem{Deng2010}
S.~Deng.
\newblock Boundedness and nonemptiness of the efficient solution sets in
  multiobjective optimization.
\newblock {\em Journal of Optimization Theory and Applications}, 144(1):29--42,
  2010.

\bibitem{Dias2021}
L.~R.~G. Dias, C.~Joi\c{t}a, and M.~Tib\v{a}r.
\newblock Atypical points at infinity and algorithmic detection of the
  bifurcation locus of real polynomials.
\newblock {\em Mathematische Zeitschrift}, 298(3-4):1545--1558, 2021.

\bibitem{Dias2017}
L.~R.~G. Dias, S.~Tanab\'e, and M.~Tib\v{a}r.
\newblock Toward effective detection of the bifurcation locus of real
  polynomial maps.
\newblock {\em Foundations of Computational Mathematics}, 17(3):837--849, 2017.

\bibitem{Dias2015}
L.~R.~G. Dias and M.~Tib\v{a}r.
\newblock Detecting bifurcation values at infinity of real polynomials.
\newblock {\em Mathematische Zeitschrift}, 279(1-2):311--319, 2015.

\bibitem{Dinh2014}
S.~T. Dinh, H.~V. H\`a, and T.-S. Ph\d{a}m.
\newblock A {F}rank--{W}olfe type theorem for nondegenerate polynomial
  programs.
\newblock {\em Mathematical Programming}, 147(1-2):519--538, 2014.

\bibitem{Ehrgott2005}
M.~Ehrgott.
\newblock {\em Multicriteria Optimization (2nd ed.)}.
\newblock Springer, Berlin, 2005.

\bibitem{Gutierrez2014}
C.~Guti\'{e}rrez, R.~L\'{o}pez, and V.~Novo.
\newblock Existence and boundedness of solutions in infinite-dimensional vector
  optimization problems.
\newblock {\em Journal of Optimization Theory and Applications},
  162(2):515--547, 2014.

\bibitem{HaHV2017}
H.~V. H\`a and T.~S. Ph\d{a}m.
\newblock {\em Genericity in polynomial optimization}.
\newblock World Scientific Publishing, Singapore, 2017.

\bibitem{Ha2006JMAA}
T.~X.~D. H\`a.
\newblock Variants of the Ekeland variational principle for a set-valued map
  involving the clarke normal cone.
\newblock {\em Journal of Mathematical Analysis and Applications},
  316(1):346--356, 2006.

\bibitem{Hartley1978}
R.~Hartley.
\newblock On cone-efficiency, cone-convexity and cone-compactness.
\newblock {\em SIAM Journal on Applied Mathematics}, 34(2):211--222, 1978.

\bibitem{Huang2004}
X.~X. Huang, X.~Q. Yang, and K.~L. Teo.
\newblock Characterizing nonemptiness and compactness of the solution set of a
  convex vector optimization problem with cone constraints and applications.
\newblock {\em Journal of Optimization Theory and Applications},
  123(2):391--407, 2004.

\bibitem{Huong2016}
N.~T.~T. Huong, J.-C. Yao, and N.~D. Yen.
\newblock Polynomial vector variational inequalities under polynomial
  constraints and applications.
\newblock {\em SIAM Journal on Optimization}, 26(2):1060--1071, 2016.

\bibitem{Jahn2004}
J.~Jahn.
\newblock {\em Vector Optimization: Theory Applications, and Extensions}.
\newblock Springer, Berlin, 2004.

\bibitem{Jelonek2014}
Z.~Jelonek and K.~Kurdyka.
\newblock Reaching generalized critical values of a polynomial.
\newblock {\em Mathematische Zeitschrift}, 276(1-2):557--570, 2014.

\bibitem{Jiao2020}
L.~G. Jiao, J.~H. Lee, and Y.~Y. Zhou.
\newblock A hybrid approach for finding efficient solutions in vector
  optimization with {SOS}-convex polynomials.
\newblock {\em Operations Research Letters}, 48(2):188--194, 2020.

\bibitem{Kim2020}
D.~S. Kim, B.~S. Mordukhovich, T.~S. Ph\d{a}m, and N.~V. Tuyen.
\newblock Existence of efficient and properly efficient solutions to problems
  of constrained vector optimization.
\newblock {\em Mathematical Programming}, 190(1-2): 259--283, 2021.

\bibitem{Kim2019}
D.~S. Kim, T.~S. Ph\d{a}m, and N.~V. Tuyen.
\newblock On the existence of {P}areto solutions for polynomial vector
  optimization problems.
\newblock {\em Mathematical Programming}, 177(1-2):321--341, 2019.

\bibitem{Lee2018}
J.~H. Lee and L.~G. Jiao.
\newblock Solving fractional multicriteria optimization problems with sum of
  squares convex polynomial data.
\newblock {\em Journal of Optimization Theory and Applications},
  176(2):428--455, 2018.

\bibitem{Lee2019}
J.~H. Lee and L.~G. Jiao.
\newblock Finding efficient solutions for multicriteria optimization problems
  with {SOS}-convex polynomials.
\newblock {\em Taiwanese Journal of Mathematics}, 23(6):1535--1550, 2019.

\bibitem{Lee2021}
J.~H. Lee, N.~Sisarat, and L.~G. Jiao.
\newblock Multi-objective convex polynomial optimization and semidefinite
  programming relaxations.
\newblock {\em Journal of Global Optimization}, 80(1):117--138, 2021.

\bibitem{Liu2021}
D.~Y. Liu, R.~Hu, and Y.~P. Fang.
\newblock Solvability of a regular polynomial vector optimization problem
  without convexity.
\newblock 2021.
\newblock {\em Optimization}, https://doi.org/10.1080/02331934.2021.1990285.

\bibitem{Luc1989}
D.~T. Luc.
\newblock {\em Theory of Vector Optimization}.
\newblock Springer-Verlag, Berlin, 1989.

\bibitem{Magron2014}
V.~Magron, D.~Henrion, and J.~B. Lasserre.
\newblock Approximating {P}areto curves using semidefinite relaxations.
\newblock {\em Operations Research Letters}, 42(6-7):432--437, 2014.

\bibitem{Magron2015}
V.~Magron, D.~Henrion, and J.~B. Lasserre.
\newblock Semidefinite approximations of projections and polynomial images of
  semialgebraic sets.
\newblock {\em SIAM Journal on Optimization}, 25(4):2143--2164, 2015.

\bibitem{Milnor1968}
J.~Milnor.
\newblock {\em Singular Points of Complex Hypersurfaces}, volume~61 of {\em
  Annals of Mathematics Studies}.
\newblock Princeton University Press, Princeton, 1968.

\bibitem{Nie2021}
J.~Nie and Z.~Yang.
\newblock The multi-objective polynomial optimization.
\newblock 2021.
\newblock arXiv:2108.04336.

\bibitem{Pham2020arXiv}
T.-S. Ph\d{a}m
\newblock Tangencies and polynomial optimization.
\newblock 2019.
\newblock arXiv:1902.06041v2

\bibitem{Sawaragi1985}
Y.~Sawaragi, H.~Nakayama, and T.~Tanino.
\newblock {\em Theory of Multiobjective Optimization}.
\newblock Academic Press, Inc., Orlando, FL, 1985.

\bibitem{Dries1996}
L.~van~den Dries and C.~Miller.
\newblock Geometric categories and o-minimal structures.
\newblock {\em Duke Mathematical Journal}, 84:497--540, 1996.

\bibitem{Yen2016}
N.~D. Yen.
\newblock An introduction to vector variational inequalities and some new
  results.
\newblock {\em Acta Mathematica Vietnamica}, 41(3):505--529, 2016.

\end{thebibliography}

\end{document}